\documentclass[a4paper,12pt]{article}

\usepackage[english]{babel}

\usepackage{amsmath,amsthm,amssymb,amsfonts,mathtools}

\usepackage{newtxtext,newtxmath}

\usepackage[skip=0.75\baselineskip plus 2pt]{parskip}
\usepackage{indentfirst}

\usepackage{xcolor}
\usepackage{etoolbox}
\usepackage{setspace}
\usepackage{enumitem}
\usepackage{titlesec}
\usepackage{bm}

\theoremstyle{plain}
\newtheorem{theom}{Theorem}[]
\newtheorem{lemma}[theom]{Lemma}
\newtheorem{cor}[theom]{Corollary}
\newtheorem{prop}[theom]{Proposition}
\theoremstyle{definition}

\newtheorem*{conj}{Conjecture}

\makeatletter
\def\thm@space@setup{
  \thm@preskip=2pt
  \thm@postskip=2pt
}
\makeatother

\makeatletter
\renewenvironment{proof}[1][\proofname] {\par\pushQED{\qed}\normalfont\topsep6\p@\@plus6\p@\relax\trivlist\item[\hskip\labelsep\bfseries#1\@addpunct{.}]\ignorespaces}{\popQED\endtrivlist\@endpefalse}
\makeatother

\usepackage{authblk}

\definecolor{auburn}{rgb}{0.43, 0.21, 0.1}
\definecolor{coolblack}{rgb}{0.0, 0.18, 0.39}
\definecolor{darkcerulean}{rgb}{0.03, 0.27, 0.49}

\usepackage{hyperref}
\hypersetup{colorlinks=true,citecolor=auburn,linkcolor=coolblack,urlcolor=darkcerulean}
\newcommand{\sref}[2]{\hyperref[#2]{#1 \ref{#2}}}
\newcommand{\zref}[2]{\hyperref[#2]{#1}}

\usepackage{geometry}
\makeatletter
\let\@fnsymbol\@arabic
\makeatother

\allowdisplaybreaks
\titlespacing{\section}{0mm}{8mm}{4mm}
\apptocmd\normalsize{%
 \abovedisplayskip=9pt
 \abovedisplayshortskip=9pt
 \belowdisplayskip=9pt
 \belowdisplayshortskip=9pt
}{}{}
\DeclareSymbolFont{largesymbolsCM}{OMX}{cmex}{m}{n}

\let\sum\relax
\DeclareMathSymbol{\sum}{\mathop}{largesymbolsCM}{"50}

\usepackage{fancyhdr}
\title{On the R\'enyi--Ulam Game with Restricted Size Queries}
\author[a]{\'Ad\'am X.\ Frakn\'oi}
\author[b,*]{D\'avid \'A.\ M\'arton}
\author[,c]{D\'aniel G.\ Simon\thanks{The author was supported by the ERC Advanced Grant "GeoScape".}}
\author[,c]{\\D\'aniel A. Lenger\thanks{The author was partially supported by the KKP 139502 grant.}}
\affil[a]{Institute of Mathematics\\E\"otv\"os Lor\'and University\\P\'azm\'any P\'eter s\'et\'any 1/C,\protect\\H-1117 Budapest\\Hungary}
\affil[b]{Bolyai Institute\\University of Szeged\\Aradi v\'ertan\'uk tere 1,\protect\\H-6720 Szeged\\Hungary}
\affil[c]{Alfr\'ed R\'enyi Institute of Mathematics\\Re\'altanoda utca 13--15,\protect\\H-1053 Budapest\\Hungary}
\affil[ ]{\protect\\[-10pt]*corresponding author: {\tt\footnotesize martondavidadam[at]gmail.com}}
\affil[ ]{\protect\\[-8pt]{\textbf{e-mails:}}\protect\\{\tt\footnotesize fraknoiadam[at]gmail.com\\martondavidadam[at]gmail.com\\ dgs45[at]cantab.ac.uk\protect\\leda1648[at]gmail.com}}
\date{}

\def\nk{\left\lfloor\frac{n}{k}\right\rfloor}
\def\nkpplus{\left\lfloor\frac{(p+1)n}{k}\right\rfloor}

\def\nkrplus{\left\lfloor\frac{(r+1)n}{k}\right\rfloor}

\def\nkqplus{\left\lfloor\frac{(q+1)n}{k}\right\rfloor}

\begin{document}

\maketitle

\vspace*{-7mm}
\begin{abstract}
\vspace*{-2mm}
\linespread{1}\selectfont
\noindent We investigate the following version of the well-known R\'enyi--Ulam game. Two players --- the Questioner and the Responder --- play against each other. The Responder thinks of a number from the set $\{1,\ldots,n\}$, and the Questioner has to find this number. To do this, he can ask whether a chosen set of at most $k$ elements contains the thought number. The Responder answers with YES or NO immediately, but during the game, he may lie at most $\ell$ times. The minimum number of queries needed for the Questioner to surely find the unknown element is denoted by $RU_\ell^k(n)$. First, we develop a highly effective tool that we call Convexity Lemma. By using this lemma, we give a general lower bound of $RU_\ell^k(n)$ and an upper bound which differs from the lower one by at most $2\ell+1$. We also give its exact value when $n$ is sufficiently large compared to $k$. With these, we managed to improve and generalize the results obtained by Meng, Lin, and Yang in a 2013 paper about the case $\ell=1$.
\\[6pt]{\itshape\bfseries Keywords:} combinatorial searching, adaptive searching, R\'enyi--Ulam liar game, small sets
\end{abstract}

\section{Introduction}
We investigate a modified version of the following combinatorial searching problem: two players --- the Questioner and the Responder --- play against each other. The Responder thinks of a number from the set $\{1,\ldots,n\}$ and the Questioner has to find this number. To accomplish that, he can ask if the sought number is an element of a chosen subset of $\{1,\ldots,n\}$, after which the Responder answers him with a YES or NO. The Questioner aims to find the thought number in as few questions as possible.

The game can be adaptive or non-adaptive. In the non-adaptive version, the Questioner must ask all his questions in advance, and the Responder answers them in the end. Whereas in the adaptive game the Questioner can ask one question at a time, which the Responder has to answer immediately.

Regarding the adaptive version, it is a well-known result that the Questioner only needs at most $\lceil\log_2 n\rceil$ rounds to find the number and that there is a case where he must ask that many questions.

R\'enyi and Ulam came up individually with the version where the Responder may lie during the game at most $\ell$ times. Let us denote the minimal number of queries needed for the Questioner in the worst case by $RU_\ell(n)$. We have already mentioned that $RU_0(n)=\lceil\log_2 n\rceil$. Beside that, the exact value of $RU_\ell(n)$ is known for all $n$ when $\ell=1,2,3$. Furthermore, for all $\ell$ there is a known lower and upper bound that differ by at most $\ell$, see \sref{Theorem}{weightbound}. For more details see \cite{du,pelc}.

There are many different versions of the above-described Basic R\'enyi--Ulam Game. We consider here the one where the Questioner can ask about subsets with at most $k$ elements, where $k<\left\lfloor\tfrac{n}{2}\right\rfloor$. Note that if $k\ge \lfloor\frac{n}{2}\rfloor$, the Questioner can use the same strategies as in the basic game, because he can ask about any of the subsets by picking either the set itself or its complement.

For this game, we denote the minimal number of questions needed in the worst case, from the Questioner's viewpoint, by $RU_\ell^k(n)$. Its exact value is known when $\ell=0$, see Katona \cite{katona1}. There it is stated that $RU_0^k(n)=\left\lfloor\tfrac{n}{k}\right\rfloor-1+\lceil\log_2(k+m_1)\rceil$ with $m_1$ denoting the remainder of $n$ divided by $k$. Katona also gave bounds for the non-adaptive version, see \cite{katona2}.

We started investigating this problem upon the suggestion of Katona, and now, using the existing bounds for the Basic R\'enyi--Ulam Game, we give a generic lower bound of $RU_\ell^k(n)$ and an upper bound that differs from the lower one by at most $2\ell+1$. Furthermore, we give its exact value for all $\ell$ when $n\gg k$.

Meng, Lin, and Yang had already published an article in 2013 concerning the case $\ell=1$, see \cite{meng}. They gave a lower bound for $RU_1^k(n)$ that differs from its value by at most $1$, and determined its exact value when $k^2\le n$. We generalize their results here for arbitrary $\ell$ and find the exact value for smaller values of $n$, too. They also made a conjecture regarding the Questioner's optimal strategy. In the last section, we disprove their conjecture with a counterexample.

Last, we would like to describe our problem as an optimal testing one. Imagine the following scenario. There is only one person who is infected with a virus, and we want to find that person. To do this, we use the so-called pooling strategy, where at most $k$ samples can be examined simultaneously. Our tests are not perfect, only relatively accurate: we assume that they make a mistake at most $\ell$ times during the process. Note that finding the optimal testing strategy here is equivalent to solving the problem of the above-described Bounded R\'enyi--Ulam Game. We also mention that the R\'enyi--Ulam Game and its many different versions have a strong connection with error correcting codes, see for example \cite{rivest}.

\section{Notation and Main Results}

Let $RU_\ell(n)$ denote the minimal number of queries needed for the Questioner to find the unknown element, assuming that he can ask an arbitrary subset of $\{1,\ldots,n\}$ and that the Responder may lie at most $\ell\ge0$ times. Let $RU_\ell^k(n)$ denote the minimum number of queries when the Questioner can only ask subsets with at most $k\ge1$ elements.

In this paper, $m_i$ denotes the remainder of $n\cdot i$ divided by $k$ for every $i\ge1$, and $\log x$ means $\log_2 x$ for every $x\in\mathbb{R}^+$.

The state of the game can be presented --- as it was introduced by Berlekamp \cite{berlekamp} --- with the vector $(x_0, x_1, \ldots, x_\ell)$, where after each answer $x_i$ denotes the number of the elements that received NO exactly $i$ times for all $0\le i\le\ell$. We say that these elements are in the $i$-th component of the game. Let $RU_\ell^k(x_0, x_1, \dots, x_\ell)$ denote the minimum number of queries needed for the Questioner from the starting state $(x_0, x_1, \ldots, x_\ell)$. Note that if a set not containing a given element receives a YES, then it counts as a NO for that element.

We denote by $x_{\ell+1}$ the number of the elements that received NO at least $\ell+1$ times. These elements can be excluded from the remaining search process because the sought number cannot be one of them. When describing a state, for clarity, we will sometimes use the notation $(x_0,\ldots,x_\ell \mid x_{\ell+1})$. Note that the game ends exactly when every element except one got a NO (at least) $\ell+1$ times, that is when $\sum_{i=0}^\ell x_i = 1$.

At the state $(x_0,\ldots,x_\ell)$, a query can also be presented by a vector $(q_0, q_1, \ldots, q_\ell)$, where $0\le q_i \le x_i$ is the number of the elements from the $i$-th component which are being queried. If the answer is YES to the question then the new state of the game is $(q_0,q_1+(x_0-q_0),\ldots,q_\ell+(x_{\ell-1}-q_{\ell-1}))$, and if it is NO then the new state is $(x_0-q_0,(x_1-q_1)+q_0,\ldots, (x_{\ell}-q_{\ell})+q_{\ell-1})$.

Given two states $(x_0,\ldots,x_\ell)$ and $(y_0,\ldots,y_\ell)$, we call the state $(x_0,\ldots,x_\ell)$ not worse --- and the state $(y_0,\ldots,y_\ell)$ not better --- if $RU_\ell^k(x_0,\ldots,x_\ell)\le RU_\ell^k(y_0,\ldots,y_\ell)$.


The following well-known lemma is a commonly used one that we will use throughout the paper. 

\begin{lemma}\label{trivieasier}
Let $\ell\ge 0,k\ge 1$. If the states $(x_0, \ldots, x_\ell)$ and $(y_0, \ldots, y_\ell)$ are such that $\sum_{j=1}^m x_j\ge \sum_{j=1}^m y_j$ for $m=1,\ldots,\ell$, then $RU^k_{\ell}(x_0,\ldots,x_\ell)\ge RU^k_{\ell}(y_0,\ldots,y_\ell)$.
\end{lemma}

Weight functions are a classical tool for the R\'enyi--Ulam game. Given a state $(x_0, \ldots, x_\ell)$, assuming that the Questioner needs $q$ queries to solve it in the worst case, its weight is $w_q(x_0,\ldots,x_\ell)\coloneqq\binom{q}{\le \ell}x_0+\binom{q}{\le \ell-1}x_1+\cdots+\binom{q}{\le 0}x_\ell$, where $\binom{q}{\le i}\coloneqq\sum_{j=0}^i \binom{q}{j}$.

It can be shown that after each question the sum of the weights of the two possible arising states is equal to the weight of the state preceding the question. Therefore the Responder can always achieve that after each query the weight decreases by no more than its half. When the Questioner finds the unknown element, the weight is $1$, since he needs zero queries then. From these follows the lower bound of \sref{Theorem}{weightbound}. The upper bound is due to Rivest et al. \cite{rivest}, for proofs see \cite{du,rivest}.

\begin{theom}\label{weightbound}
Let $\ell \ge 0$, $n\ge 1$, and define the \emph{weight bound} as
\[W_\ell(n)\coloneqq\min\left\{ q:n\le \frac{2^q}{\binom{q}{\le\ell}}\right\}.\]
Then
$W_\ell(n)\le RU_\ell(n)\le W_\ell(n)+\ell.$
\end{theom}

Note that the above-defined weight bound is the same as the well-known Hamming bound in the field of coding theory.

Now we present the most important tool we use throughout the paper, which we call the Convexity Lemma.

\begin{lemma}\label{convex2}(Convexity Lemma)
Let $\ell\ge1$, $1\le i\le j \le \ell$, such that $x_i,x_j\ge1$, and let $1\le a \le \min \{i,\ell-j+1\}$. Then, except for the case where $\sum_{i=0}^{\ell} x_i=2$ and $a=\ell-j+1$,
\begin{multline*}
RU_\ell^k(x_0, \dots, x_i, \dots, x_j, \dots , x_\ell \mid x_{\ell+1}) \\\le RU_\ell^k(x_0, \dots, x_{i-a}+1, \dots, x_i-1, \dots, x_j-1, \dots, x_{j+a}+1, \dots, x_\ell\mid x_{\ell+1}).
\end{multline*}
\end{lemma}

The above lemma allows us to prove our main results, which are presented in the following theorem --- the main theorem of the article.

\begin{theom}\label{maintheorem}
Let $\ell\ge0$, $k,n\ge1$, such that $1<k<\left\lfloor\frac{n}{2}\right\rfloor$, and let
\[L_\ell^k(n)\coloneqq\max_{0\le p\le \ell}\left(\nkpplus-1+RU_{\ell-p}(k)\right).\]
Then it is true that
\begin{enumerate}[itemsep=-2pt,topsep=-2pt]
\item[(i)] {\center $L_\ell^k(n)\le RU_{\ell}^k(n)\le L_\ell^k(n)+\ell+1.$}
\item[(ii)] Furthermore if $\ell\ge1$, $n\ge\max\{k(\ell+5), k(\log\log 2k+6), k(\log\log 2k+2\log \ell+5)\}$ then
\[RU_\ell^k(n)=\left\lfloor\frac{(\ell+1)n}{k}\right\rfloor-1+\left\lceil\log (k+m_{\ell+1})\right\rceil.\]
\end{enumerate}
\end{theom}

We note that $RU_\ell^1(n)=(\ell+1)n-1$ for arbitrary $\ell \ge 0$, $n \ge 1$, and --- as already mentioned --- $RU_0^k(n)=\left\lfloor\tfrac{n}{k}\right\rfloor-1+\lceil\log(k+m_1)\rceil$ for all $1\le n$, $1\le k \le n$.

We also note that the exact value of $RU_\ell(k)$ is unknown for $\ell\ge4$. For an always calculable bound we can apply \sref{Theorem}{weightbound}, which yields the following.

\begin{cor}\label{maintheorem-corollary}
Let $\ell \ge 0$, $k, n\ge 1$, such that $1<k<\left\lfloor\frac{n}{2}\right\rfloor$, and
\[\tilde{L}_\ell^k(n):=\max_{0\le p\le \ell}\left(\nkpplus-1+W_{\ell-p}(k)\right).\]
Then $\tilde{L}_\ell^k(n)\le RU_\ell^k(n)\le \tilde{L}_\ell^k(n)+2\ell+1$.
\end{cor}

To conclude this section, we present bounds for $W_\ell(n)$. Similar results can be found in \cite{rivest} --- without determining the constants. The proof is a straightforward calculation that can be found in \sref{Appendix}{appendix}.

\begin{lemma}\label{ru-estimate}
For every $\ell\ge1, n\ge2$ we have that
\[\log n + \ell\log\log 2n-\ell\log \ell\le W_{\ell}(n)\le \log n +\ell\log\log 2n+2\ell\log \ell+2,\]
and therefore, as a consequence of \sref{Theorem}{weightbound},
\[\log n + \ell\log\log 2n-\ell\log \ell\le RU_{\ell}(n)\le \log n +\ell\log\log 2n+2\ell\log \ell+\ell+2.\]
\end{lemma}

\section{The Convexity Lemma}

In this section, we prove the \zref{Convexity Lemma}{convex2} and present some of its immediate consequences. Before the proof we will need the following auxiliary lemma.

\begin{lemma}\label{endofgame}
Consider either the Basic or the Bounded R\'enyi--Ulam Game for some $\ell,k\ge1$ and for any starting state $(x_0,\ldots,x_\ell)$ with $\sum_{m=0}^\ell x_m\ge2$. Let the Questioner play a certain strategy that uses $q$ queries in its worst case. Then if, along this strategy, the game ends with the state $(x_0, \ldots,x_{i-1} , x_i, x_{i+1} \ldots, x_\ell)=(0,\ldots, 0, 1, 0 \ldots, 0)$, we know that the Questioner used at most $q-(\ell-i)$ queries, where $0\le i\le \ell$.
\end{lemma}

\begin{proof}
Let $r$ denote the number of queries used in the case described above. Then the penultimate state after $r-1$ questions was either $(0,\dots,0,1,0,\dots,a)$ or $(0,\dots,1,0,0,\dots,a)$ for some $a\ge1$.

We want to show that from here the Responder can always achieve a play where the Questioner needs to ask at least $\ell-i+1$ more. We consider the state $(0,\ldots,0,1,0,\ldots,a)$, since it is clearly not worse for the Questioner. We will show that from here the Responder can achieve, after an arbitrary query, that the next state is one of the followings: $(0,\ldots,0, 0,1,\ldots, b)$ or $(0,\ldots,0, 1,0,\ldots, a)$, where $0<b\le a$.

If the Questioner asks all the elements then the Responder can achieve the second case with a YES. If the Questioner asks $(0,\ldots,0, 1,0,\ldots, a-b)$ or $(0,\ldots,0, 0,0,\ldots, b)$ then the Responder can achieve the first case with a NO or a YES, respectively.

Using the same argument $\ell-i$ times, we conclude that the Responder can achieve a not better state than $(0,\ldots, c)$, where $c\ge 2$. From here the Questioner needs at least one more question.

Now it follows that $(r-1)+(\ell-i+1)\le q$, meaning the Questioner used at most $q-(\ell-i)$ queries, indeed.
\end{proof}

\begin{proof}[Proof of the \zref{Convexity Lemma}{convex2}]
Let us call Game 1 and Game 2 the two games starting from the states $(x_0, \dots, x_i, \dots x_j, \dots , x_\ell \mid x_{\ell+1})$ and $(x_0, \dots, x_{i-a}+1, \dots, x_i-1, \dots, x_j-1, \dots, x_{j+a}+1, \dots, x_\ell \mid x_{\ell+1})$, respectively.

Let us think about these games in the following way. There are two elements of $\{1,\ldots,n\}$, say $A$ and $B$, that are in the $i$-th and $j$-th component of the starting state of Game 1, respectively. In the starting state of Game 2, $A$ is transferred to the $(i-a)$-th, and $B$ is transferred to the $(j+a)$-th component.

We will show that every questioning strategy for Game 2 that uses $q$ queries in its worst case can be converted into a questioning strategy for Game 1 with the same property.

Let us play the two games in parallel, such that in Game 1 we are the Questioner and we have an opponent, called Opponent 1, who is the Responder, while in Game 2 we are the Responder, and our other opponent, called Opponent 2, is the Questioner.

Assume that Opponent 2 plays a questioning strategy that uses $q$ queries in its worst case. What we want is to convert this into a questioning strategy for Game 1. We do this using the following method.

First, we wait for the question of Opponent 2. We can assume that Opponent 2 does not ask any element that is in the $(\ell+1)$-th component of Game 2, otherwise we just ignore it. Then, accordingly, we ask in Game 1 as defined next.

We ask exactly the same elements as Opponent 2, except for $A$ and $B$.  If Opponent 2 asks both $A$ and $B$, then we do too, and if he only asks one of them, we ask the one being in the component with the smaller index --- if both are in the same component, let us ask $B$. Otherwise, we ask none of them. Note that this method is well-defined and that we use as many queries in Game 1 as Opponent 2 uses in Game 2.

Second, we wait for the answer of Opponent 1 to the constructed query, and then answer the same in Game 2.

Let $s_1(A)$, $s_1(B)$ denote the indices of the components containing $A$ and $B$ in Game 1, and $s_2(A)$, $s_2(B)$ in Game 2. Note that these indices can be even $\ell+1$, but no more due to our assumption. At the beginning of the game we have that $s_1(A)=i$, $s_1(B)=j$, $s_2(A)=i-a$, $s_2(B)=j+a$. It follows from the definition of our method that the equality
\begin{equation}\label{eq:index1}
s_1(A)+s_1(B)=s_2(A)+s_2(B)
\end{equation}
and the inequalities
\begin{equation}\label{eq:index2}
\min\{s_2(A),s_2(B)\}\le s_1(A)\le s_1(B) \le\max\{s_2(A),s_2(B)\}
\end{equation}
always hold. We note that for any other element, the index of its component is equal in Game 1 and Game 2.

From \eqref{eq:index1} and \eqref{eq:index2}, if at least one of $A$ or $B$ is not in the $(\ell+1)$-th component of Game 2 then at least one of them is not in the $(\ell+1)$-th component of Game 1. Furthermore, if neither $A$ nor $B$ fall out from the search --- meaning they got less than $\ell+1$ NO --- up to a point in Game 2, then none of them did in Game 1.

Opponent 2 wins when $\sum_{m=0}^\ell x_m=1$ in Game 2, which he surely achieves after at most $q$ queries. When he wins, one of the following cases occurs.

\textit{Case 1:} Both $A$ and $B$ are in the $(\ell+1)$-th component of Game 2. According to \eqref{eq:index1}, then they are in the $(\ell+1)$-th component of Game 1, too. Thus $\sum_{m=0}^\ell x_m=1$ in Game 1, meaning we have won there.

\textit{Case 2:} One of $A$ or $B$ is in the $(\ell+1)$-th, while the other one is in the $\ell$-th component of Game 2. Using \eqref{eq:index1}, this is clearly true for Game 1, meaning $\sum_{m=0}^\ell x_m=1$ there, too.

\textit{Case 3:} One of $A$ or $B$ is in the $(\ell+1)$-th, while the other one is in the $t$-th component of Game 2, where $t\le\ell-1$. According to \sref{Lemma}{endofgame}, this can happen after at most $q-(\ell-t)$ rounds. Note that we can actually use \sref{Lemma}{endofgame} for Game 2, since it is true that $\sum_{m=0}^\ell x_m\ge2$ at the starting state of Game 2 --- that is why we had to exclude the case when $a=\ell-j+1$ and $\sum_{m=0}^{\ell} x_m=2$ for the starting state of Game 1.

Now we have that $1\le \sum_{m=0}^\ell x_m\le2$, meaning we have not necessarily won in Game 1. However, from \eqref{eq:index1} and \eqref{eq:index2}, we also know that either $s_1(A)=t$ or $s_1(A)>t$, where the former one yields that $\sum_{m=0}^\ell x_m=1$. When the latter occurs, we can surely win in Game 1 after querying $A$ $\ell-t$ times. Therefore, we only need at most $q$ queries in this case, too. 

Assuming that Opponent 2 plays an optimal questioning strategy --- which uses exactly \[RU_\ell(x_0, \dots, x_{i-a}+1, \dots, x_i-1, \dots, x_j-1, \dots, x_{j+a}+1, \dots, x_\ell \mid x_{\ell+1})\] queries --- in Game 2, we get the statement of the lemma.
\end{proof}

We call a query \emph{optimal} for the Questioner at the state $(x_0,\ldots,x_\ell)$ if he can find the sought number from the possible resulting states using at most $RU_\ell^k(x_0,\ldots,x_\ell)-1$ queries. The following consequence of the \zref{Convexity Lemma}{convex2} presents --- under a certain condition --- an optimal query that will be very useful in the next section.

\begin{lemma}\label{nyakatekertlemma}
Consider the Bounded R\'enyi--Ulam Game for some $\ell\ge0$, $k,n\ge1$. Suppose that the Questioner knows that the Responder will answer NO to his next question when he is at the state $(x_0,\ldots,x_\ell\mid x_{\ell+1})$. Then the query $\bm{q}=(q_0,\ldots,q_\ell)$ with $q_i\coloneqq\min\left\{x_i, k-\sum_{j=0}^{i-1} q_j\right\}$ for all $0\le i\le \ell$ is optimal.
\end{lemma}

\begin{proof}
    Take an optimal question $\bm{r}=(r_0,\ldots,r_\ell)$ for the conditions described above such that the $\ell_1$-distance $\Vert\bm{q}-\bm{r}\Vert_1$ is as small as possible. Now indirectly suppose that $\Vert\bm{q}-\bm{r}\Vert_1 \neq 0$.
    
    By the definition of $\bm{q}$, the smallest index $i$ for which $q_i \neq r_i$ must have $q_i>r_i$. If $q_i \ge r_i$ for all $0\le i\le \ell$, then because of \sref{Lemma}{trivieasier} $\bm{q}$ is not a worse question than $\bm{r}$, so it is optimal, contradicting our indirect assumption.
    
    Therefore, there has to exist an index $i< j\le\ell$ with the property $r_j>q_j$. We will show that the query $\hat{\bm{r}} \coloneqq (r_0,\ldots,r_{i}+1,\ldots,r_{j}-1,\ldots,r_\ell)$ is optimal for the investigated case, too.
    
    The question $\bm{r}$ leads to the state
    $\bm{y}\coloneqq (x_0-r_0,x_1-r_1+r_0,\ldots,x_\ell-r_\ell+r_{\ell-1} \mid x_{\ell+1} + r_\ell)$,
    while $\hat{\bm{r}}$ leads to the state $\hat{\bm{y}}\coloneqq\bm{y} - \bm{e}_i + \bm{e}_{i+1} + \bm{e}_{j} - \bm{e}_{j+1}$, where $\bm{e}_m$ denotes the $m$-th standard basis vector for all $0\le m\le \ell+1$. Due to the \zref{Convexity Lemma}{convex2}, $\hat{\bm{y}}$ is not worse than $\bm{y}$, meaning that $\hat{\bm{r}}$ is also optimal, indeed. However, $\Vert\bm{q}-\hat{\bm{r}}\Vert_1 < \Vert\bm{q}-\bm{r}\Vert_1$, which is a contradiction.
\end{proof}

We conclude this section with the following fairly intuitive lemma. Although we will not use it later, its proof demonstrates the utility of the \zref{Convexity Lemma}{convex2} and its consequence.

\begin{lemma}\label{onemorelie}
For all $\ell\ge0$, $k,n\ge1$ it is true that $RU^k_{\ell+1}(n)\ge RU^k_{\ell}(n)+\nk$.
In other words, if one more lie is available for the Responder, then at least $\nk$ more questions are required for the Questioner.
\end{lemma}
\begin{proof}
Consider the game with $\ell+1$ lie, and assume that the Responder answers NO for the first $\nk$ questions. Then we know from \sref{Lemma}{nyakatekertlemma} that the Questioner arrives at a state that is not better than $(m_1, n-m_1, 0, \dots, 0)$, since --- for him --- the optimal question will always be $(k,0,\ldots,0)$. This is a not better state than $(0,n,0,\dots, 0)$, which requires $RU_\ell^k(n)$ questions to solve.
\end{proof}

\section{Proof of the Bounds}
In this section, we show that $L_\ell^k(n)$ is a lower bound of $RU_\ell^k(n)$, and then we prove that they differ by at most $\ell+1$.
\begin{proof}[Proof of \sref{Theorem}{maintheorem} \textit{(i)}]
First, let us prove that \[L_\ell^k(n)\coloneqq\max_{0\le p\le \ell}\left(\nkpplus-1+RU_{\ell-p}(k)\right)\] is a lower bound. Imagine ourselves in the Questioner's place, and assume that we know that the Responder will answer NO for the first $\nkpplus-1$ questions for some $0\le p \le \ell$. Then, from \sref{Lemma}{nyakatekertlemma}, we know that we will get to a state that is not better than $(0,\dots,k+m_{p+1},n-k-m_{p+1},\dots,0)$. Even from this state, we need at least
\begin{align}\label{eq:biz1}
\begin{split}
RU_\ell^k(0,\dots,k+m_{p+1},n-k-m_{p+1},\dots,0)&\ge RU_\ell^k(0,\dots,k+m_{p+1},0,\dots,0)\\
&\ge RU_{\ell-p}^k(k)=RU_{\ell-p}(k)
\end{split}
\end{align}
queries. It follows that $\nkpplus-1+RU_{\ell-p}(k)$ is a lower bound for every $0\le p\le \ell$, indeed.

Second, we prove the upper bound by induction on $\ell$ with fixed values of $k$ and $n$. To be precise, we will prove that the above-mentioned lower bound differs from the exact value by at most $\ell+1$.

For the base case $\ell=0$ we need $RU^k_0(n)-L_0^k(n)=\lceil\log{(k+m_1)}\rceil-\lceil\log{k}\rceil\le 1$, which clearly holds true. Now, for the induction step, we must prove that if $RU_{\ell-p}^k(n)-L_{\ell-p}^k(n)\le \ell-p+1$ for all $1\le p\le \ell$ then $RU_{\ell}^k(n)-L_\ell^k(n)\le \ell+1$. 

With $\nk$ queries, using \sref{Lemma}{trivieasier}, we can always reach a state that is not worse than $(k, n-k, 0, \dots, 0)$ if we always ask $k$ elements from the zeroth component. From this state, consider the set of those $k$ elements that are currently at the zeroth component and always ask this set until one of the following cases occur.

\textit{Case 1:} The $n-k$ elements from the first component fall out from the searching, while the $k$ elements from the zeroth arrive at the $p$-th component, where $0\le p\le\ell-1$. This results the state $(0\dots,0, k,0, \dots,0)$ after altogether $\nk+\ell+p$ questions. From here, we still need $RU_{\ell-p}(k)$ queries.

Now we define $a_{\ell,p}\coloneqq(p+1)\nk-1+RU_{\ell-p}(k)$ and $b_{\ell,p}\coloneqq\nk+\ell+p+RU_{\ell-p}(k)$, where $0\le p\le \ell-1$. It is clear that $a_{\ell,p}\le L_\ell^k(n)$ for all $p$. Furthermore, $\max_{0\le p \le\ell-1} b_{\ell,p}$ is an upper bound for the questions needed altogether in Case 1, because here we can find the unknown element with that many queries. Let us now compare $L_\ell^k(n)$ and $b_{\ell,p}$. Then we see that, since $\nk\ge2$,
\[b_{\ell,p}-L_\ell^k(n)\le b_{\ell,p}-a_{\ell,p}=p\cdot\left(1-\nk\right)+\ell+1\le \ell+1.\]
Therefore $b_{\ell,p}\le L_\ell^k(n)+\ell+1$ for every $p$, thus $L_\ell^k(n)+\ell+1$ is an upper bound in this case.

\textit{Case 2:} The $k$ elements of the zeroth component merge with the $n-k$ elements of the first one at the $p$-th component, where $1\le p\le \ell$. This results the state $(0\dots,0, n,0, \dots,0)$ after altogether $\nk+2p-1$ questions. From here, we still need $RU^k_{\ell-p}(n)$ queries.

Now, similarly as above, we define $c_{\ell,p}\coloneqq p\cdot\nk+L_{\ell-p}^k(n)$ and $d_{\ell,p}\coloneqq\nk+2p-1+RU_{\ell-p}^k(n)$ for all $1\le p\le \ell$. It is clear that $\max_{1\le p\le \ell}d_{\ell,p}$ is an upper bound for the questions needed altogether in Case 2. We will also show later that for all $1\le p\le \ell$
\begin{equation}\label{eq:c}
c_{\ell,p}\le L_\ell^k(n).
\end{equation}
Using \eqref{eq:c} and the induction hypothesis we have that, since $\nk\ge2$,
\[d_{\ell,p}-L_\ell^k(n)\le d_{\ell,p}-c_{\ell,p}\le (1-p)\cdot\left(\nk-1\right)+\ell+1\le \ell+1,\]
meaning $L_\ell^k(n)+\ell+1$ questions are enough in this case, too.

The proof is now almost complete, it only remains to verify \eqref{eq:c}. That is, we want to show that if $\ell\ge1$ then $p\cdot\nk+L_{\ell-p}^k(n)\le L_\ell^k(n)$ holds for all $1\le p\le \ell$. After substitution, we want to have that
\[p\cdot\nk+\max_{0\le r\le \ell-p}\left(\nkrplus-1+RU_{(\ell-p)-r}(k)\right) \le \max_{0\le r\le \ell}\left(\nkrplus-1+RU_{\ell-r}(k)\right).\]
Comparing the appropriate members of the maximums --- for all $0\le r\le \ell-p$ compare the $r$-th term from the left-hand side with $(r+p)$-th one from the right-hand side ---, we get that
\[p\cdot \nk+\nkrplus-1+RU_{(\ell-p)-r}(k)\le\left\lfloor\frac{(r+p+1)n}{k}\right\rfloor-1+RU_{\ell-(r+p)}(k)\]
for all $0\le r\le \ell-p$, which yields that \eqref{eq:c} is indeed true, concluding our proof.
\end{proof}

\section{Proof of the Exact Value}\label{sectionexact}
In this section, we sketch a strategy for the Questioner, which --- assuming that $n$ is sufficiently large with respect to $k$ --- uses as many questions as our lower bound, hence proves the second part of \sref{Theorem}{maintheorem}.

\begin{proof}[Proof of \sref{Theorem}{maintheorem} \textit{(ii)}]
As seen before in \eqref{eq:biz1}, $\max_{0\le p\le \ell}\left(\left\lfloor\tfrac{(p+1)n}{k}\right\rfloor-1+RU_{\ell-p}(k+m_{p+1})\right)$ is a lower bound of $RU_\ell^k(n)$. Therefore $\left\lfloor\tfrac{(\ell+1)n}{k}\right\rfloor-1+RU_0(k+m_{\ell+1})=\left\lfloor\tfrac{(\ell+1)n}{k}\right\rfloor-1+\left\lceil\log (k+m_{\ell+1})\right\rceil$ is a lower bound, too.

Imagine ourselves in the Questioner's place, and let us play the following questioning strategy. We ask $k$ elements from the components with the smallest possible indices until we receive a YES. That is, we ask $k$ elements from the first non-zero component if that is possible, otherwise, we ask all of the elements there and the remaining ones from the next component.

It is clear that if we always receive a NO until the $\left(\left\lfloor\tfrac{(\ell+1)n}{k}\right\rfloor-1\right)$-th question, the state of the game is $(0,\ldots,0,k+m_{\ell+1})$. Here we switch to an optimal strategy for the case $\ell=0$. Thus we will use at most $RU_0^k(k+m_{\ell+1})=RU_0(k+m_{\ell+1})=\left\lceil\log (k+m_{\ell+1})\right\rceil$ questions. Therefore $\left\lfloor\tfrac{(\ell+1)n}{k}\right\rfloor-1+\left\lceil\log (k+m_{\ell+1})\right\rceil$ queries are indeed enough in this case. 

Moreover, if we get a YES earlier, but when the first non-zero component is already the $\ell$-th one, we arrive at the state $(0,\ldots,0,k)$ --- since the $k$ elements being asked remain in the $\ell$-th component, while the other elements fall out from the search. This state is solvable with at most $\lceil\log k\rceil$ questions, meaning we need less than $\left\lfloor\frac{(\ell+1)n}{k}\right\rfloor-1+\left\lceil\log (k+m_{\ell+1})\right\rceil$ questions.

There are two ways of getting a YES even earlier --- up to the $\left\lceil\tfrac{\ell\cdot n}{k}\right\rceil$-th question. Let $0\le s\le\ell-1$ denote the index of the first non-zero component before the question receives the first YES.

\textit{Case A:} There are at least $k$ elements in the $s$-th component. Then $k$ elements remain in the $s$-th component, and all the elements of the $(s+1)$-th go to the $(s+2)$-th. All the other elements go to the $(s+1)$-th component. So we arrive at the state \[(x_0,\ldots,x_{s-1},x_s,x_{s+1},x_{s+2},x_{s+3},\ldots,x_\ell)=(0,\ldots,0,k,x,n-k-x,0,\ldots,0),\] where $m_{s+1}\le x\le n-k$ and $x\equiv m_{s+1}$ modulo $k$. From here, we ask about the $x$ elements in the $(s+1)$-th component, querying only $k$ at a time when possible and $m_{s+1}$ for the last time. It is actually enough to consider only this strategy since for these plus questions the NO answer always results in the not better state.
Indeed, using the \zref{Convexity Lemma}{convex2}, we see that if $k\le x\le n-k$, \[(x_0,\ldots,x_{s-1},x_s,x_{s+1},x_{s+2},x_{s+3},x_{s+4},\ldots,x_\ell)=(0,\ldots,0,0,2k,x-k,n-k-x,0,\ldots,0)\] is not worse than $(0,\ldots,0,k,x-k,n-x,0,0,\ldots,0)$. Similarly, if $x=m_{s+1}$ then $(0,\ldots,0,0,k+m_{s+1},0,n-k-m_{s+1},0,\ldots,0)$ is not worse than $(0,\ldots,0,k,0,n-k,0,0,\ldots,0)$. So we can conclude that in this case --- using altogether $\left\lceil\tfrac{(s+1)n}{k}\right\rceil$ queries --- we arrive at a state that is not worse than $(0,\ldots,0,k,0,n-k,0,0,\ldots,0)$.

\textit{Case B:} There are less than $k$ --- thus $m_{s+1}$ --- elements in the $s$-th component, meaning we get the first YES for the $\left\lceil\tfrac{(s+1)n}{k}\right\rceil$-th question. Note that in this case $m_{s+1}>0$. The resulting state is $(0,\ldots,0,m_{s+1},k-m_{s+1},n-k,0,\ldots,0)$, which is not worse than $(0,\ldots,0,k,0,n-k,0,\ldots,0)$.

In summary, if we get the first YES after a query that is preceded by a state in which the first non-zero component is the $s$-th one, we need at most $\left\lceil\tfrac{(s+1)n}{k}\right\rceil+RU_\ell^k(0,\ldots,0,k,0,n-k,0\ldots,0)$ questions, where $0\le s\le \ell-1$.

Consider the state $(0,\ldots,0,k,0,n-k,0\ldots,0)$. From here we always ask the set that contains those $k$ elements that currently form the $s$-th component. After each answer either the $k$ or the $n-k$ elements move to the next component. We do this until one of the following occurs.

\textit{Case 1:} The elements from the $s$-th component merge with the elements from the $(s+2)$-th in the $p$-th component, where $s+2\le p\le \ell$. This can happen after $2p-2s-2$ queries. The resulting state is solvable with at most $RU_{\ell-p}^k(n)$ questions, which --- according to the first part of the \zref{theorem}{maintheorem} --- is no more than $\max_{0\le q\le \ell-p}\left(\nkqplus-1+RU_{\ell-p-q}(k)\right)+\ell-p+1$.

\textit{Case 2:} The elements from the $(s+2)$-th component fall out from the search, and the elements from the $s$-th arrive at the $p$-th component, where $s\le p\le \ell-1$. That can happen after $\ell+p-2s-1$ queries. The resulting state is solvable with at most $RU_{\ell-p}(k)$ questions.

In the first case, using \sref{Lemma}{ru-estimate}, we can bound the queries needed altogether in the following way:
\begin{align*}
\begin{split}
&\left\lceil\frac{(s+1)n}{k}\right\rceil+2p-2s-2+\max_{0\le q \le \ell-p}\Bigg(\left\lfloor\frac{(q+1)n}{k}\right\rfloor-1+\log k+(\ell-p-q)\log\log 2k\\
&\quad+2(\ell-p-q+1)\log(\ell-p-q+1)+(\ell-p-q)+2\Bigg)+\ell-p+1
\end{split}\\
\begin{split}
\le\text{ }&\frac{(s+1)n}{k}+2\ell-2s+1+\log k+(\ell-p)\log\log 2k\\
&\quad+\max_{0\le q\le \ell-p}\Bigg(\frac{(q+1)n}{k}-q\log\log 2k+2(\ell-p-q+1)\log(\ell-p-q+1)-q\Bigg).
\end{split}
\end{align*}
Let us think of $q$ as a real number and of the expression in the latter maximum as its continuous function on $[0,\ell-p]$. Consider the derivative of this function. It is $\tfrac{n}{k}-\log\log 2k-2\log(\ell-p-q+1)-\tfrac{2}{\ln 2}-1$, so at least $\tfrac{n}{k}-\log\log 2k-2\log \ell-4$, which is non-negative if $n\ge k(\log\log 2k+2\log \ell+4)$. So when $n$ satisfies this condition, our upper bound is maximal if $q=\ell-p$.

Substituting this, we get the upper bound \[\frac{(s+1)n}{k}+\frac{(\ell-p+1)n}{k}+\log k+\ell+p-2s+1.\] It is clear --- since $\tfrac{n}{k}\ge 1$ --- that with respect to $p$ this expression is monotonically decreasing, therefore it is maximal when $p=s+2$. Substituting this, we get $\tfrac{\ell\cdot n}{k}+\log k+\ell-s+3$, which is clearly maximal when $s=0$.

Therefore we use at most $\tfrac{\ell\cdot n}{k}+\log k+\ell+3$ queries in the first case. Compare this to $\tfrac{(\ell+1)n}{k}-2+\log k$, which is a lower bound of $\left\lfloor\tfrac{(\ell+1)n}{k}\right\rfloor-1+\left\lceil\log (k+m_{\ell+1})\right\rceil$. We can see that if $n\ge k(\ell+5)$ then $\tfrac{\ell\cdot n}{k}+\log k+\ell+3\le\tfrac{(\ell+1)n}{k}-2+\log k$. Thus in this case $\left\lfloor\tfrac{(\ell+1)n}{k}\right\rfloor-1+\left\lceil\log (k+m_{\ell+1})\right\rceil$ questions are indeed enough.

Very similarly, in the second case we can estimate the number of queries, using \sref{Lemma}{ru-estimate}, in the following way:
\begin{align*}
&\left\lceil\frac{(s+1)n}{k}\right\rceil+\ell+p-2s-1+\log k+(\ell-p)\log\log 2k+2(\ell-p)\log(\ell-p)+(\ell-p)+2\\[4pt]
\le\text{ }&\frac{(s+1)n}{k}+\log k+2\ell-2s+2+(\ell-p)\log\log 2k+2(\ell-p)\log(\ell-p).
\end{align*}
It is clear that the latter one is monotonically decreasing with respect to $p$, so we substitute $p=s$ and get
\begin{equation}\label{eq:bound}
\frac{(s+1)n}{k}+\log k+2\ell-2s+2+(\ell-s)\log\log 2k+2(\ell-s)\log(\ell-s).
\end{equation}
Now consider this as a continuous function of $s$ on $[0,\ell-1]$. Its derivative is then \[\frac{n}{k}-2-\log\log 2k-2\log(\ell-s)-\frac{2}{\ln 2}\ge\frac{n}{k}-\log \log 2k-2\log \ell-5.\] The latter one --- so the derivative, too --- is non-negative if $n\ge k(\log\log 2k+2\log \ell+5)$. Therefore when $n$ satisfies this condition, \eqref{eq:bound} is monotonically increasing with respect to $s$, so it is maximal if $s=\ell-1$. Using this, it is now enough to have that $\tfrac{\ell\cdot n}{k}+\log k+\log\log 2k+4\le\frac{(\ell+1)n}{k}-2+\log k$, which is true if $n\ge k(\log\log 2k+6)$. Thus $\left\lfloor\frac{(\ell+1)n}{k}\right\rfloor-1+\left\lceil\log (k+m_{\ell+1})\right\rceil$ questions are enough in the second case, too.

In summary, we saw that if $n\ge\max\{k(\ell+5), k(\log\log 2k+6), k(\log\log 2k+2\log \ell+5)\}$, the Questioner can find the sought number using at most $\left\lfloor\frac{(\ell+1)n}{k}\right\rfloor-1+\left\lceil\log (k+m_{\ell+1})\right\rceil$ queries, which is also a lower bound of $RU_\ell^k(n)$, concluding our proof.
\end{proof}

\section{Additional remarks}

In this section, we present some additional results. We will discuss a bound that is more accurate than that of \sref{Theorem}{maintheorem}, but before that, we prove an easy consequence of \sref{Lemma}{trivieasier}.

\begin{lemma}\label{egykerdeskulonbseg}
Let $\ell\ge0$, $n\ge2$ and $x,y\le n$, such that $0<x<y\le2x$. Then $$RU_\ell(y, n-y, 0, \dots, 0)-RU_\ell(x, n-x, 0, \dots, 0)\le 1$$
\end{lemma}
\begin{proof}
At the state $(y, n-y, 0, \dots, 0)$ let the Questioner ask $(y-x, 0, \dots, 0)$. Then the NO answer results exactly the state $(x, n-x, 0, \dots, 0)$, whereas the YES answer results the state $(y-x,x,n-y,\ldots,0)$, which is not worse for the Questioner by \sref{Lemma}{trivieasier}, since $y\le 2x$.
\end{proof}

\begin{prop}
With the conditions of \sref{Theorem}{maintheorem}, let \[\hat{L}_\ell^k(n)\coloneqq\max_{0\le p\le \ell}\left(\nkpplus-1+RU_{\ell-p}(k+m_{p+1}, n-k-m_{p+1},0,\dots, 0)\right).\] Then
$\hat{L}_\ell^k(n)\le RU_\ell^k(n)\le \hat{L}_\ell^k(n)+\ell$.
\end{prop}

Note that if we substitute $\ell=1$, we get a lower bound with a difference from the real value at most $1$, just as Meng, Lin, and Yang \cite{meng}. However, $RU_{\ell-p}(k+m_{p+1},n-k-m_{p+1},0,\dots,0)$ is unknown in the general case, just as $RU_{\ell-p}(k)$ --- furthermore, there is no known useful bound of it.

The proof goes by division into cases, just as the proof of \sref{Theorem}{maintheorem} described in detail --- the only plus tool needed is \sref{Lemma}{egykerdeskulonbseg}. Thus, we omit it from our paper.

Now we examine the problem of finding the optimal strategy for the Questioner. As we have seen in \sref{Section}{sectionexact}, we can give an optimal strategy when the number of the elements, $n$ is sufficiently large. However, this strategy fails to be optimal when $n$ is smaller.

Indeed, consider the value $\left\lfloor\frac{\ell\cdot n}{k}\right\rfloor-1+RU_1(k)$, which --- according to \sref{Theorem}{maintheorem} --- is a lower bound of $RU_\ell^k(n)$. Using \sref{Lemma}{ru-estimate}, if $n<k(\log\log 2k-3)$,
\begin{align*}
\left\lfloor\frac{\ell\cdot n}{k}\right\rfloor-1+RU_1(k)&> \frac{\ell\cdot n}{k}-2+\log k+\log\log 2k\\
&>\frac{(\ell+1)n}{k}+1+\log k\ge \left\lfloor\frac{(\ell+1)n}{k}\right\rfloor-1+\left\lceil\log (k+m_{\ell+1})\right\rceil.
\end{align*}
Therefore, if $n<k(\log\log 2k-3)$ then, as seen above, $RU_\ell^k(n)>\left\lfloor\frac{(\ell+1)n}{k}\right\rfloor-1+\left\lceil\log (k+m_{\ell+1})\right\rceil$. This also means that for the described strategy --- when $n$ is relatively small --- an earlier YES leads to more questions --- unlike for larger values of $n$.

However, the following lemma states that our strategy is optimal up to the $\left(\left\lfloor\frac{n}{k}\right\rfloor-1\right)$-th question.

\begin{lemma}\label{bestquestion}
If the Bounded R\'enyi--Ulam Game --- for some $\ell\ge0$, $k,n\ge1$ --- is at a state $(x_0,\ldots,x_\ell)$ with $x_0\ge 2k$ then $(k,0,\ldots,0)$ is an optimal query for the Questioner, with the NO answer resulting the not better state.
\end{lemma}
\begin{proof}
Assume that the Questioner asks the query $(q_0,\ldots,q_\ell)$. Then the resulting state after a YES is $(q_0,q_1+x_0-q_0,\ldots,q_\ell+x_{\ell-1}-q_{\ell-1})$. Using the facts that $\sum_{i=0}^\ell q_i\le k$ and $x_0\ge 2k$, it is easy to see that $(x_0-q_0,x_1-q_1+q_0,\ldots,x_\ell-q_\ell+q_{\ell-1})$ is a not better state. Indeed, using the \zref{Convexity Lemma}{convex2} repeatedly --- shifting back every questioned element except for the $q_0$ elements questioned from the zeroth component while shifting right $q_1+\cdots+q_\ell$ elements from the zeroth component --- we see that
\begin{align*}
RU_\ell^k(x_0-q_0,x_1-q_1+q_0,\ldots,x_\ell-q_\ell+q_{\ell-1})\ge &
RU_\ell^k\left(x_0-\sum_{i=0}^\ell q_i,x_1+\sum_{i=0}^\ell q_i,x_2,\ldots,x_\ell\right),
\intertext{by shifting right $x_0-q_0-\sum_{i=0}^\ell q_i$ many elements from the zeroth component,}
\ge& RU_\ell^k(q_0,x_1+x_0-q_0,x_2,\ldots,x_\ell),
\intertext{by shifting right $x_i - q_i$ many elements from each component with index $1\le i \le \ell$,}
\ge & RU_\ell^k(q_0,q_1+x_0-q_0,\ldots,q_\ell+x_{\ell-1}-q_{\ell-1}),
\end{align*}
where we used \sref{Lemma}{trivieasier}.

Therefore it suffices to consider the case where the Responder surely answers NO. From this it follows --- using \sref{Lemma}{nyakatekertlemma} --- that $(k,0,\ldots,0)$ is an optimal question, indeed.
\end{proof}

After the $\left(\left\lfloor\frac{n}{k}\right\rfloor-1\right)$-th question it is not clear what should the Questioner do when $n$ is relatively small. It seems almost impossible to give an optimal question for an arbitrary state --- even in the case $\ell=1$. Meng, Lin, and Yang conjectured the following in \cite{meng} about the optimal query.

\begin{conj}
Let $k\ge1$, $a,b\ge0$, such that $2k>a$, $a+b\ge 2$, and $\chi(a,b)\coloneqq\min\big\{q : (q+1)a+b\le 2^q\big\}$. Using the notations
\begin{align*}
C&\coloneqq\max\big\{\chi(a,b), RU_0^k(2a+b)\big\},\\
\chi_0&\coloneqq\max\big\{x : RU_1^k(x,a-x)\le C-1, x\le a, x\le k\big\},\\
\chi_1&\coloneqq\max\big\{y : RU_1^k(\chi_0,y+a-\chi_0)\le C-1,y\le b, y\le k-\chi_0\big\},
\end{align*}
it is true that $RU_1^k(a,b)=\max\big\{C,1+RU_1^k(a-\chi_0,b-\chi_1+\chi_0)\big\}$.
\end{conj}

We have found a counterexample using a recursive program\footnote{\url{https://github.com/fraknoiadam/bounded-renyi-ulam-game}} that calculates all the optimal questioning strategies for small values of $k$ and $n$. It showed that with $k=16$ and $n=56$ the Questioner reaches the state $(10,44)$ after three optimal queries. This state is solvable within seven questions, and for $a=10$, $b=44$ we have that $C=7$. Therefore the question $(\chi_0,\chi_1)$ has to be optimal, where we know that $(\chi_0,\chi_1)=(8,6)$. However, the only optimal query for this state is $(7,9)$. This example also shows that our strategy is not optimal for small values of $n$.

\section*{Acknowledgements}
We would like to thank here Gyula Katona for proposing the problem. We are also tremendously grateful to D\"om\"ot\"or P\'alv\"olgyi for helping us as a supervisor, and for organizing the REU 2020 program, which provided us with the opportunity to research.

\appendix

\section{The proof of \sref{Lemma}{ru-estimate}}\label{appendix}

\begin{proof}[Proof of \sref{Lemma}{ru-estimate}]
First, let us consider the upper bound. We estimate the denominators in the weight bound's definition as 
\begin{equation}\label{qestimate}
\binom{q}{0}+\dots+\binom{q}{\ell}\le 1+q+\dots + q^\ell \le  (q+1)^\ell,
\end{equation}
Then we take the following modified version of $W_\ell(n)$,
where we now take the minimum for non-negative real values of $q$: \[W'_\ell(n)\coloneqq\min\left\{ q:n\le\frac{2^q}{(q+1)^\ell}\right\}.\]
Note that by \eqref{qestimate} we have that $W_\ell(n)\le \lceil W_\ell'(n) \rceil\le W_\ell'(n)+1$. Our goal now is to give an upper bound for $W'_\ell(n)$, meaning we want to find a real number $q$ satisfying $n\le\tfrac{2^q}{(q+1)^\ell}$. Let us search it in the form $q=\log{n}+\log{x}$, where $x>1$. Then it is enough to find an $x$ with \[n\le \frac{nx}{(\log{n}+\log{x}+1)^\ell},\]
or, equivalently, with $(\log{n}+\log{x}+1)^\ell\le x$.

If $x=\left((\ell+1)\log n+\ell^2+1\right)^\ell$, this condition is fulfilled. 
Indeed, after substituting we get \[(\log n+\ell\log((\ell+1)\log n+\ell^2+1)+1)^\ell\le((\ell+1)\log n+\ell^2+1)^\ell,\] which can be simplified to \[\log((\ell+1)\log n+\ell^2+1)\le\log n+\ell,\] and then to \[(\ell+1)\log n+\ell^2+1\le2^\ell\cdot n,\] with the last one being clearly true for all $n\ge2$, $\ell\ge1$. Thus we have that
\begin{align*}
W_\ell(n)&\le W'_\ell(n) +1\le\log n+\ell\cdot\log((\ell+1)\log n+\ell^2+1)+1\\[4pt]
&\le \log n+\ell\log(\ell^2(\log n+1))+2\\[4pt]
&=\log n+\ell\log\log 2n+2\ell\log \ell+2.
\end{align*}
Now, let us consider the lower bound. We claim that the following estimation holds for the denominators in the definition of the weight bound when $\ell\ge2$:
\begin{equation}\label{sum-qi}
\sum_{i=0}^\ell \binom{q}{i}\ge 1+\sum_{i=0}^{\ell}\frac{\binom{\ell}{i}q^i}{\ell^\ell}=1+\frac{(q+1)^\ell}{\ell^\ell}.
\end{equation}
Indeed, compare the members of the sums for $i\ge 2$ respectively using \[\binom{q}{i}=\frac{q}{i}\cdot \frac{q-1}{i-1}\cdots \frac{q-i+1}{1}\ge \left(\frac{q}{i}\right)^i\ge \frac{q^i}{\ell^i}=\frac{\ell^{\ell-i}q^i}{\ell^\ell}\ge \frac{\binom{\ell}{\ell-i}q^i}{\ell^\ell}=\frac{\binom{\ell}{i}q^i}{\ell^\ell}.\]
Using this, it is sufficient to show that \[\binom{q}{0}+\binom{q}{1} \ge 1+\frac{\binom{\ell}{0}q^0}{\ell^\ell}+\frac{\binom{\ell}{1}q^1}{\ell^{\ell}}\] to prove \eqref{sum-qi}. This is equivalent to $q\ell(\ell^{\ell-1}-1)\ge1$, which is true for all $\ell\ge2$, since $q\ge1$.

Now, because we are searching for a lower bound for the weight bound, we look for a real number $q$ satisfying \[n\ge \frac{2^q}{\frac{(q+1)^\ell}{\ell^\ell}+1}.\] Similarly as above, we search it in the form $q=\log{n}+\log{x}$, where $x>1$. Substituting this, we get the condition $x\le \frac{(\log n +\log x +1)^\ell}{\ell^\ell}+1$. Let $x\coloneqq\left(\frac{\log 2n}{\ell}\right)^\ell+1$, then we get \[\frac{(\log n+1)^\ell}{\ell^\ell}+1\le \frac{(\log n + \log x +1)^\ell}{\ell^\ell}+1,\] which is clearly fulfilled, since $x>1$.

From this reasoning it follows that \[W_\ell(n)\ge \log{n}+\log{\left(\frac{\log 2n}{\ell}\right)^\ell}=\log n + \ell\log\log 2n-\ell\log \ell.\]
We still need to consider the case $\ell=1$. For $q=\log n+ \log\log 2n$ we have that $n\ge \frac{2^q}{q+1}$, which yields that $\log n+\log\log 2n \le W_1(n)$, indeed.
\end{proof}

\end{document}